\documentclass[reqno]{amsart}
\usepackage[usenames, dvipsnames]{color}
\usepackage{amsthm}
\usepackage{amsmath}
\usepackage{amssymb}
\usepackage{amscd}
\usepackage{graphics}
\usepackage{latexsym}
\usepackage{stmaryrd}
\usepackage{hyperref}
\usepackage{empheq}

\theoremstyle{remark}
\newtheorem{Remark}{Remark}[section]

\definecolor{titlecol}{named}{BrickRed}
\definecolor{headcol}{named}{Violet}
\definecolor{seccol}{named}{Red}
\definecolor{sseccol}{named}{Bittersweet}
\definecolor{pbcol}{named}{Black}
\definecolor{sncol}{named}{Brown}
\definecolor{acol1}{named}{Red}
\definecolor{acol2}{named}{Apricot}

\def\p{\partial}
\def\R{\mathbb{R}}

\def\l{\lambda}

\def\cH{{\mathcal H}}

\def\cK{{\mathcal K}}

\def\cM{{\mathcal M}}

\theoremstyle{plain}
\newtheorem{thm}{Theorem}
\newtheorem{conj}{Conjecture}

\newtheorem{prop}[thm]{Proposition}
\newtheorem{lemma}[thm]{Lemma}

\newcommand{\ddc}{\sqrt{-1}\partial\bar\partial}

\begin{document}

\title{On the uniqueness of even $L_p$-Minkowski problem}
\dedicatory{Dedicated to Professor Xiuxiong Chen on his 60th birthday}

\author{Weiyong He}
\address{Department of Mathematics, University of Oregon, Eugene, OR 97403.}
\email{whe@uoregon.edu}
\author{Junbang Liu}
\address{Department of Mathematics, Hong Kong University of Science and Technology, Clear Water Bay.}
\email{junbangliu@ust.hk}

\begin{abstract}We prove that there is a unique $p_0\in [0,1)$, which can be characterized by the eigenvalue of Hilbert operator related to a convex body, that the even $L^p$ Minkowski problem has a unique solution for $p\geq p_0$, and the uniqueness fails for infinitely many convex bodies if $p<p_0$. The previous results by many experts in the field assert that the uniqueness holds for $p>p_0$. \end{abstract}
\maketitle

\section{Introduction}

The Minkowski problem forms the core of various areas in fully nonlinear
partial differential equations and convex geometry (see Trudinger, Wang \cite{TW} and Schneider \cite{Schneider}), which was extended to the $L_p$-Minkowski theory
by Lutwak \cite{Lut93, Lut94, Lut95} where $p = 1$ corresponds to the classical case. The
classical Minkowski’s existence theorem due to Minkowski and Aleksandrov
characterizes the surface area measure $S_K$ of a convex body $K$ in $\R^{n+1}$, more
precisely, it solves the Monge-Amp\'re equation
\[
\det(\nabla^2 h+h\; \text{Id})=f
\]
on the sphere $S^{n}$ where a convex body $K$ with $C^2_+$
boundary provides
a solution if $h = h_K|_{S^{n}}$ for the support function $h_K$ of $K$, and in this
case, $1/f(x)$ is the Gaussian curvature at the $y \in \p K$ where $x$ is an exterior normal of $x\in S^{n}$. One may consult Cheng-Yau \cite{CY} and Pogorelov \cite{P} for a full discussion on the Minkowski problem and its resolution. 
The so-called log-Minkowski problem
\begin{equation}
    h \det(\nabla^2 h+ h\text{Id})=f
\end{equation}
or $L_0$-Minkowski problem was posed by Firey \cite{Firey66} in his seminal paper, which seeks to characterize the cone volume measure 
\[
dV_K=\frac{1}{n+1} h_KdS_K
\]
of a convex body $K$ containing the origin $o\in \R^{n+1}$, and to determine whether the even solution is unique if $f$ is even. The latter problem is the so-called Log-Minkowski conjecture by Lutwak. The log-Minkowski problem has received due attention only after Lutwak's $L_p$ Minkowski problem
\begin{equation}\label{p-minkowski-01}
    h^{1-p}\det (\nabla^2h+h\text{Id})=f
\end{equation}
in the 1990's where $p=1$ and $p=0$ are the classical and the logarithmic Minkowski problem. 

The uniqueness of the solution of the classical Minkowski problem up to translation is the consequence of the equality case of the Brunn-Minkowski inequality
\begin{equation}
    V(\alpha L+\beta K)^{\frac{1}{n+1}}\geq \alpha V(L)^{\frac{1}{n+1}}+\beta V(K)^{\frac{1}{n+1}}
\end{equation}
for convex bodies $L$ and $K$, $\alpha, \beta \geq 0$. For $p\geq 0$, the $L_p$ Minkowski problem is intimately related to the $L_p$ version of the conjectured Brunn-Minkowski inequality 
\begin{equation}\label{BM-01}
    V((1-\lambda)L+_{p}\lambda K)\geq \left((1-\lambda)V(L)^{\frac{p}{n+1}}+\lambda V(K)^{\frac{p}{n+1}}\right)^{\frac{n+1}{p}}
\end{equation}
for $\lambda\in (0, 1)$ and convex bodies $L, K$ containing the origin. The case $p=0$ is interpreted as the limit case
\[
 V((1-\lambda)L+_{0}\lambda K)\geq V(L)^{1-\lambda} V(K)^{\lambda}.
\] When $p=1$, \eqref{BM-01} is the classical Brunn-Minkowski inequality and a theorem of Firey \cite{Firey64}, for $p>1$, where he introduced the $L_p$ Minkowski sum of convex bodies for $p>1$. The notion was then extended by B\"or\"oczky, Lutwak, Yang and Zhang \cite{BLYZ01} for $p\in [0, 1)$. 
Assuming $L, K$ are origin symmetric, the even $L_p$ Brunn-Minkowski conjecture states that \eqref{BM-01} holds for $p\in [0, 1)$ \cite[Problem 1.1, 1.5]{BLYZ01}, in which the conjecture was established for $n=1$. The conjecture is known to be equivalent to the even $L_p$-Minkowski inequality (conjecture) \cite[Lemma 3.1]{BLYZ01}.
\begin{conj}[B\"or\"oczky-Lutwak-Yang-Zhang, Even $L_p$ Brunn-Minkowski conjecture]
For any origin-symmetric convex bodies in $\R^{n+1}$, \eqref{BM-01} holds for $p\in [0, 1)$. Moreover, it is equivalent to 
\begin{equation}\label{BM-02}
    \frac{1}{p}\int_{S^n} h_L^p h_K^{1-p} dS_K\geq \frac{n+1}{p}V(K)^{1-\frac{p}{n+1}}V(L)^{\frac{p}{n+1}}.
\end{equation}
The case $p=0$ is interpreted as the limit case, known as the conjectured Log-Minkowski inequality,
\begin{equation}
    \frac{1}{V(K)}\int_{S^n} \log\frac{h_L}{h_K} dV_K\geq \frac{1}{n+1} \log \frac{V(L)}{V(K)}.
\end{equation}
\end{conj}
This family of inequalities has attracted much interest, see for instance \cite{ CLM, GZ, LMNZ, Ma, Rotem, Sar01, Sar02, XL, KM01, CHLL, Put, BK, Milman23, Xi} etc. 
The conjectured inequalities above for even convex bodies are essentially equivalent to the following $L_p$-Minkowski conjecture,
\begin{conj}
    For $p\in [0, 1)$, the Monge-Ampere equation on $S^n$
    \[
    h^{1-p}\det(\nabla^2 h+h\text{Id})=f
    \]
    has a unique even solution if $f$ is a positive even smooth function on $S^n$. 
\end{conj}

There are some substantial progress for the conjectures. When $f=1$, the uniqueness problem receives special attention since the work of Firey \cite{Firey66} and its relation with the Gauss curvature flow \cite{Chou85, Andrews99, AGN}. The uniqueness is proved for $p\in (-n-1, 1)$ without assuming the even condition by Brendle, Choi and Daskalopoulos \cite{BCD} when $f$ is a constant. For general even function $f$, there are also considerate important progress towards the uniqueness conjecture,  notably by Kolesnikov and Milman \cite{KM01} and Chen, Huang, Li, Liu \cite{CHLL}.
We introduce necessary notations. Let $K$ be a compact convex set in $\mathbb{R}^{n+1}$. The support function of $K$ is defined as $h_K:\mathbb{R}^{n+1}\rightarrow \mathbb{R}, h_K(x)=\max_{y\in K}x\cdot y$. Let $u\in S^n$, then we call $u$ is an outer unit normal of $K$ at a boundary point $y$ if $y\cdot u=h_K(u)$. A boundary point is said to be singular if it has more than one unit normal vector, and it's well-known that the set of singular points has Hausdorff measure $\cH^{n}$ equal to 0. Let $\nu_K:\partial K\rightarrow S^n$ be the generalized Gauss map. For each Borel set $\omega\in S^n$, define $S_K(\omega):=\cH^n(\nu_K^{-1}(\omega))$. The measure $S_K$ on $S^n$ is called the Aleksandrov-Fenchel-Jessen surface area measure of $K$. For any two convex bodies $K,L\in \mathbb{R}^{n+1}$, $\lambda\in (0,1)$, $p\in (0,\infty)$, one defines the $L_p$-Minkowski sum $(1-\lambda)L+_p\lambda K$ as \begin{equation}
  (1-\lambda)L+_p\lambda K:=\cap_{x\in S^n}\{y\in\mathbb{R}^{n+1}:y\cdot x\leq (1-\lambda)h_L^p(x)+\lambda h_K^p(x)\}.  
\end{equation}  
The limit case $p=0$ is interpreted as \begin{equation}
     (1-\lambda)L+_0\lambda K:=\cap_{x\in S^n}\{y\in\mathbb{R}^{n+1}:y\cdot x\leq h_L^{1-\lambda}(x)h_K^{\lambda}(x)\}.  
\end{equation}
In the following, we use $\cK$ to denote the set of all convex bodies in $\mathbb{R}^{n+1}$, and let $\cK_e$ be the set of origin-symmetric bodies in $\cK$. $\cK_{+,e}^2$ is the set of convex bodies with $C^{2}$-boundary and strictly positive curvature, in other words, the support function of convex body $K$ in $\cK_{+,e}^2$ is a $C^2$ even function with $\nabla^2h_K+h_K\text{Id}>0$.

Kolesnikov and Milman \cite{KM01} investigated a local version of $L_p$ Brunn-Minkowski inequality and local uniqueness of $L_p$ Mikowski problem. 
For each smooth strictly convex body $K$, the Hibert-Brunn-Mikowski operator is introduced in \cite{KM01} as follows
\begin{equation}
    L_K z= h^{-1}H_K^{ij}(h_k^2z_i)_j,
\end{equation}
where $H_K^{-1}$ is the inverse of the matrix $H_K:=\nabla^2h+h\text{Id}$. 
For each $K\in \cK_{+}^2$, the operator $L_K$ is an self-adjoint, semipositive definite  elliptic operator in $L^2(dV_K)$. Here are some basic facts (see \cite[Theorem 5.3]{KM01})
 about the eigenvalues of the Hilbert-Brunn-Minkowski operator,
 \begin{enumerate}
     \item $0$ is an eigenvalue of  $-L_k$, and the eigenspace consists of constant functions $E_0$. 
      \item $-L_k\geq n\text{Id}$ when it is restricted on $E_0^\perp$, which is equivalent to the Brunn-Minkowski equality. 
     \item $n$ is an eigenvalue of $-L_k$ with multiplicity precisely $n+1$ corresponding to the $(n+1)$ dimensional subspace $E_1^K$ spanned by normalized linear functions
     \[
     E_{1}^K=\text{span}\{h_K^{-1}(u) \langle u, v\rangle, v\in \R^{n+1}\}.
     \]
 \end{enumerate}
Kolesnikov and Milman \cite{KM01} formulated the following local $L^p$-Minkowski conjecture for a symmetric convex body $K\in \cK_{+, e}^2$ (we use slighly different notations). 
\begin{conj}[Local $L_p$-Minkowski conjecture, Kolesnikov-Milman]\label{locallpminkowski}
    Let $n\geq 2$ and $p\in [0,1)$. For all $K\in \cK_{+,e}^2$, and $ z\in C^2_{e}(S^n) $ with $\int_{S^n}zh_KdS_K=0$, one has \begin{equation}\label{hessianestimate}
\int_{S^n}H_K^{ij}z_iz_jh_K^2 dS_K\geq (n+1-p)\int_{S^n}z^2h_KdS_K,
\end{equation}where $(H_K^{ij})$ is the inverse of $\nabla^2h_K+h_K\text{Id}$.
\end{conj}
In particular, when $p=0$, this gives the local log-Minkowski conjecture, see \cite{CLM} as well.
\begin{conj}[Local log-Minkowski conjecture, Kolesnikov-Milman]
    When restricted on the subspace $E_{\text{even}}$ of even functions, the eigenvalue of $-L_K$ satisfies 
\[
\lambda_{1, e}(-L_K)\geq n+1. 
\]
\end{conj}

A main result proved by Kolesnikov and Milman \cite{KM01} is the following, 
\begin{thm}[Kolesnikov-Milman]
There exists $p_n\in (0, 1)$ of the form $p_n=1-cn^{-3/2}$ such that, for all $K\in \cK_{+, e}^2$, 
\[
\lambda_{1, e}>n+1-p_n.
\]
Moreover suppose \eqref{hessianestimate} holds for some fixed $n, p\in[0,1)$, and $K\in \cK_{+,e}^2$, then for any $q>p$, there exists a $C^2$-neighborhood $N_{q,K}$ of $K$ in $\cK_{+,e}^2$, such that if  $L_1,L_2\in N_{q,K}$ satisfy $h_{L_1}^{1-q}dS_{L_1}=h_{L_2}^{1-q}dS_{L_2}$, then $h_{L_1}=h_{L_2}$.
\end{thm} 
By recent progress on the KLS conjecture  \cite{ChenY} and results in \cite{KM01}, one can get that $p_n$ is of the form $1-c_\epsilon n^{-1-\epsilon}$ see \cite{Milman23, B24} for more details. 
Relying on the local uniqueness theorem above,  Chen, Huang, Li, Liu \cite{CHLL} proved $L_p$ Minkowski conjecture for $p\in (p_n, 1)$ using the continuity method in elliptic PDE theory, and they also proved the uniqueness when $f$ is sufficiently closed to $1$ (or a constant) and even. 
\begin{thm}[Chen-Huang-Li-Liu]
    Let $p\in (p_n, 1)$. the Monge-Ampere equation on $S^n$
    \[
    h^{1-p}\det(\nabla^2 h+h\text{Id})=f
    \]
    has a unique even solution if $f$ is a positive even smooth function on $S^n$. 
\end{thm}
 When $f$ is closed to a constant but not assumed to be even, there are some recent progress as well \cite{CFL, BS}.
In this paper we continue to investigate the uniqueness of the even $L_p$ Minkowski problem, after the progress in \cite{KM01, CHLL, Put}. Let $p_0\in [0, 1)$ such that
\begin{equation}
    \inf_{K\in \cK^2_{+, e}} \lambda_{1, e}(-L_K)=n+1-p_0.
\end{equation}

Our main theorem can be stated as follows,
\begin{thm}\label{main}
    The even $L_p$ Minkowski conjecture holds for $p=p_0$; that is, 
     \[
    h^{1-p_0}\det(\nabla^2 h+h\text{Id})=f
    \]
    has a unique smooth even solution if $f$ is a positive even smooth function on $S^n$.  Moreover, for any $K\in \cK^4_{+, e}$, we have
    \begin{equation}\label{eigenvalue-01}
    \lambda_{1, e}(-L_K)>n+1-p_0. 
    \end{equation}
    For each  $p\in (-n-1, p_0)$, there exists infinitely many positive smooth even functions $f$, the equation
    \[
    h^{1-p}\det(\nabla^2 h+h\text{Id})=f
    \]
    has at least two even smooth solutions. 
\end{thm}
The non-uniqueness result in Theorem \ref{main} was derived by E. Milman in \cite[Corollary 1.5]{Milman22} as a direct consequence of the fact that $\inf\lambda_{1, e}\leq n+1$.  We include the statement in our theorem for completeness. 
Chen, Huang, Li, Liu \cite{CHLL} proved that the even $L_p$ Minkowski conjecture holds for $(p_0, 1)$. Our improvement is to prove that the uniqueness conjecture actually holds at $p_0$ as well. In short we can say that the uniqueness is closed if it is open in $p$.  
Our result is a consequence of the following,  
\begin{thm}Suppose for $p\in [0, 1)$, the $L_p$-Brunn-Minkowski inequality \eqref{BM-02} holds, then 
\[
h^{1-p}\det(\nabla^2h+h\text{Id})=f
\]
has a unique smooth even solution if $f$ is a positive even smooth function on $S^n$. 
\end{thm}
The equivalence between the even $L_p$ Brunn-Minkowski conjectures and the even $L_p$ Minkowski conjecture are well-known, see \cite{BLYZ01} and  in particular \cite[Theorem 1.1, Theorem 2.1]{Milman23} for the most recent update. 
In particular Putterman \cite[Proposition 1.5]{Put} proved that a local version of $L_p$-Brunn-Minkowski inequality implies the global version of the inequality, verifying a conjecture in \cite{KM01}. Her results in particular imply that 
the $L_p$-Brunn-Minkowski inequalities hold at $p=p_0$. 
Our result above (Theorem 4) provides one missing piece in the picture of the equivalence of these conjectures. 
The other new ingredient of our result is to provide the strict eigenvalue estimate \eqref{eigenvalue-01}. We actually prove the following,
\begin{thm}
    For $K\in \cK^4_{+, e}$, then $\lambda_{1, e}(-L_K)$ cannot have a local minimum. 
\end{thm}
As this manuscript was being completed, we note that S. Wang also obtained similar results in\cite{Wang25}, where he proved that if the local log Brunn-Minkowski inequality holds, then the equality occurs only for homothetic convex bodies.  It's equivalent to that $\lambda_{1,e}(-L_K)>0$. However, our approach here is very different from \cite{Wang25}.  
Kolesnikov and Milamn \cite[Proposition 5.14]{KM01} conjectured that $\inf \lambda_{1, e}$ is taken by the cube,  which is known to be $n+1$ \cite[Theorem 10.2]{KM01} (note that we use a different normalization and dimension convention).  Our result shows that if the conjecture $\inf \lambda_{1, e}=n+1$ is true, then indeed $\lambda_{1, e}(-L_K)>n+1$ for any $K\in \cK^4_{+, e}$. 
 For $p<0$, it is known that both \eqref{BM-01} and \eqref{BM-02} fail in general and there are many previous known non-uniqueness results \cite{Stancu02, Andrews03, CW06, JLW, HLW16, Li19, LLL22, KM01} etc. The $L_p$ Brunn-Minkowski theory and $L_p$ Minkowski problem have attract much attention over last several decades after the seminal work of Lutwak \cite{Lut95}, see for instance \cite{BBCY, BHZ, BLYZ02, CLZ, CW06, HLYZ, HLX, LYZ, Stancu01, U, Zhu01, Zhu02} etc. 
The existence and uniqueness of the $L_p$ Minkowski problem are well understood when $p\geq 1$. When $p<0$, the uniqueness fails even restricted to smooth origin-symmetric convex bodies \cite{JLW}.  When $0\leq p<1$,  the uniqueness of solutions to the $L_p$-Minkowski problem turns out to be really challenging and it remains a very important open problem in the field.
We refer in particular to the recent survey paper by K. B\"or\"oczky \cite{B24} and \cite{KM01, Milman23} for an overview and more references of the problem and related topics in the subject. 

The work of Kolesnikov and Milman \cite{KM01} indicates the significance of the study of the first even eigenvalue $\lambda_{1, e}(-L_K)$ for the uniqueness conjecture of the  even Logarithm Minkowski problem. 
Built on the progress in \cite{KM01, CHLL, Put}, our results above give a characterization of the uniqueness to the even $L_p$ Minkowski conjecture.
In \cite{Milman22}, E. Milman provides an elegant proof  that the maximum of $\lambda_{1, e}(-L_K)$ equals $2(n+1)$, and the maximum is achieved precisely by all the ellipsoids. The study of infimum turns out to be much more delicate. For nonsmooth convex body $K$,  $-L_K$ does not have an apparent meaning as a regular differential operator. There are several notions of $\lambda_{1, e}$ introduced for nonsmooth case \cite{KM01, Put, Milman22}, see in particular \cite[Section 1.4]{Milman22} for a detailed discussion.  It is proved \cite{KM01} that $\lambda_{1, e}(-L_K)$ is invariant under the action of a general linear transformation of $\R^{n+1}$,  \[\lambda_{1, e}(-L_{T(K)})=\lambda_{1, e}(-L_K), T\in GL(n+1).\]
Here we give an alternative definition of the first even eigenvalue when a convex body is not smooth. 
We introduce a complex notation. We embed $\R^{n+1}$ into $\R^{n+1}\times \mathbb{T}^{n+1}$, where $\mathbb{T}^{n+1}$ is a standard torus.
Using the coordinate $x\in \R^{n+1}$ and $y\in \mathbb{T}^{n+1}$, we have a standard complex structure on $\R^{n+1}\times \mathbb{T}^{n+1}$ with the coordinate $z=x+\sqrt{-1}y$. 
The support function $h$, as a homogeneous convex function on $\R^{n+1}$, can be viewed as to a plurisubharmonic function on $\R^{n+1}\times \mathbb{T}^{n+1}$, simply extending $h$ to be constant along each $\mathbb{T}^{n+1}$ fiber. For a given function $u$ on $S^n$, it can also be extended as a homogenous function in $\R^{n+1}$ and hence a function on $\R^{n+1}\times \mathbb{T}^{n+1}$. Then we have the following,
\[
H^{ij}u_iu_jh^2 \det(H) (\sqrt{-1})^{n+1} dZ\wedge d\bar Z=c_n\sqrt{-1}\partial u\wedge \bar\partial u\wedge B^n,
\]
where $B$ is the $(1, 1)$ form on $\R^{n+1}\times \mathbb{T}^{n+1}$,
\[
B=\sqrt{-1}\partial\bar\partial h+\sqrt{-1}\partial r\wedge \bar \partial r,
\]
and $r=|x|$ is the radial function on $\R^{n+1}$, extended to $\R^{n+1}\times \mathbb{T}^{n+1}$ trivially in the $\mathbb{T}^{n+1}$ fiber. 
As a direct consequence of well known facts in complex pluripotential theory, 
\[
\sqrt{-1}\partial u\wedge \bar\partial u\wedge B^n 
\]
depends continuously on $h$ (in the $L^\infty$ topology). 
As a direct application, for any $u\in C^2_e(S^n)$, we have a well defined Radon measure for a convex body in $\cK_e$,
\[
\mathfrak{m}_{K, u}=H^{ij}\det{H} u_iu_j dS.
\]
 Given this, we can directly extend the definition of $\lambda_{1, e}(-L_K)$ for $K\in \cK_e$,  see \cite[Section 1.4]{Milman22} for comparison. 
It is rather straightforward to see that $\lambda_{1, e}(-L_K)$ is upper semicontinuous with respect to the Hausdorff topology.
The following problem seems to be really challenging and intriguing. 
\begin{conj}
    The first even eigenvalue $\lambda_{1, e}(-L_K)$ is continuous with respect to the Hausdorff topology of convex bodies. In other words, if $h_k$ converges to $h$ in $L^\infty$, then $\lambda_{1, e}(-L_{K_k})$ converges to $\lambda_{1, e}(-L_K)$. 
\end{conj}

{\bf Acknowledgment}: the first author is grateful to Yong Huang for valuable discussions on the uniqueness problem related to the log-Minkowski problem.

\numberwithin{thm}{section}
\numberwithin{equation}{section}

\section{The Minkowski problem and variational approach}
The variational approach for the $L_p$ Minkowski problem, $-n-1<p$ was introduced in the seminal work of Chou and Wang \cite{CW06}.
We will focus on the setting $p\in [0, 1)$. 
Let $f$ be a positive even smooth function on $S^n$. 
We define the probability measure induced by $f$ as \[d\mathfrak{m}_f=f dS \left(\int_{S^{n}} fdS\right)^{-1}.\]
For an origin symmetric convex body $K$ with support function $h$, 
we write $dS_K=dS_h=\det(\nabla^2h+h\text{Id})dS_x$ to be the surface area measure.
More generally, for $p\in [0, 1]$, $dS_{h, p}=h^{1-p}\det(\nabla^2 h+h\text{Id})dS_x$. We also write the corresponding probability measure as
\[
d\mathfrak{m}_{K, p}=dS_{h, p} \left(\int_{S^n} dS_{h, p}\right)^{-1}. 
\]
Now let $h$ be an even smooth support function of a convex body. 
First we recall the $F(0, f, h)=F_{0, f}(h)$ functional 
\begin{equation}
   F(0, f, h)=F_{0, f}(h)=\int_{S^{n}} (\log {h})d\mathfrak{m}_f-\frac{1}{n+1}\log V(h),
\end{equation}
and for each $p$ and $f$, recall the definition of $F(p, f, h)=F_{p ,f}(h)$ by
\begin{equation}
F(p, f, h)=F_{p, f}(h)=\frac{1}{p}\log\int_{S^{n}} h^p d\mathfrak{m}_f-\frac{1}{n+1} \log V(h).
\end{equation}
When the probability measure $d\mathfrak{m}_f$ is given by $d\mathfrak{m}_{K, p}$, the functional is denoted as $F_{p, K}$.
The first variation formula is given by \begin{equation}\label{1stvariation-01}
\nabla_h F_{p, f}(u)=\int_{S^n} u\left(h^{p-1}d\mathfrak{m}_f\left({\int_{S^n}h^pd\mathfrak{m}_f}\right)^{-1}-\frac{dS_h}{(n+1)V(h)}\right).
\end{equation}
So the critical points satisfy 
\begin{equation}\label{lp-02}
\left({\displaystyle\int_{S^{n}}h^{p} d\mathfrak{m}_f}\right)^{-1}h^{p-1} d\mathfrak{m}_f=\frac{\det(\nabla^2 h+h \text{Id})dS_x}{(n+1)V(h)}.    
\end{equation}
The following existence theorem is substantial for us \cite[Theorem D, Section 5]{CW06}.
\begin{thm}
    [Chou-Wang]Let $f\in C^\alpha_e(S^n)$ be a positive even function. For any $p\in (-n-1, 1)$, then $F_{p, f}(h)$ attains a global minimum and it satisfies the equation \eqref{lp-02}. Moreover, $h\in C^{2, \alpha}_{+, e}$ and the corresponding convex body  is in $\cK^{2, \alpha}_{+, e}$. 
\end{thm}
The second variation of $F$ is computed in \cite{Milman22}. We include here for completeness and we use slghtly different notations. 
The second variation at a critical point is given by \[
\begin{aligned}
    \nabla^2F(h)(h_t,h_t)
    =&\int_{S^n}h_t\frac{d}{dt}\left(h^{p-1}d\mathfrak{m}_f\left({\int_{S^n}h^pd\mathfrak{m}_f}\right)^{-1}-\frac{dS_h}{(n+1)V(h)}\right)\\
    =&\frac{(p-1)}{\int_{S^n}h^pd\mathfrak{m}_f}\int_{S^n}h^{p-2}h_t^2d\mathfrak{m}_f-\frac{p}{(\int_{S^n}h^pd\mathfrak{m}_f)^2}\left(\int_{S^n}h_th^{p-1}d\mathfrak{m}_f\right)^2\\
    &+\frac{1}{(n+1)V(h)}\int_{S^n}H^{ij}h_{ti}h_{tj}dS_h-\frac{1}{(n+1)V(h)}\int_{S^n}H^{ij}\delta_{ij}h_t^2dS_h\\
    &+\frac{1}{(n+1)V(h)^2}\left(\int
_{S^n}h_tdS_h\right)^2\\
=&\frac{(p-1-n)}{(n+1)V(h)}\int_{S^n}z^2hdS_h\\
&+(n+1-p)\left(\int
_{S^n}\frac{zhdS_h}{(n+1)V(h)}\right)^2+\frac{1}{(n+1)V(h)}\int_{S^n}H^{ij}z_{i}z_{j}h^2dS_h,
\end{aligned}
\]where $z=h^{-1}h_t$. 
Note that the functional $F(\lambda h)=F(h)$ for any $\lambda>0$, hence it is natural to consider the convex bodies with a fixed volume, by fixing the scaling. Hence we can consider the convex bodies with the volume fixed,
\[
V(h)=\frac{1}{n+1}\int_{S^n} h\det (\nabla^2 h+h \text{Id})dS=V(B_1)=\frac{\omega_n}{n+1},
\]
where $\omega_n$ is the surface area of $S^n\in \R^{n+1}$. The normalization condition, for $z=h^{-1} h_t$
\begin{equation}\label{normal-002}
    \int_{S^n} z h_KdS_K=0
\end{equation}
is the precisely the variation with the volume fixed.
Hence $\nabla^2F(h)(z,z)\geq 0$ at a critical point $h$, is equivalent to for all $z\in C^{2}_e(S^n)$ satisfying \eqref{normal-002}, one has \[
\int_{S^n}H^{ij}z_iz_jh^2dS_h\geq (n+1-p)\int_{S^n}z^2hdS_h.
\]
The first even eigenvalue $\lambda_{1, e}(-L_K)$ is then defined as follows,
\[
\lambda_{1, e}=\inf_{u\in C^2_e(S^n)} \frac{\displaystyle \int_{S^n}H^{ij}h^2u_iu_j\det(H)dS}{\displaystyle\int_{S^n}{(u-\underline u)^2} h\det(H)dS},
\]
where $u-\underline u \neq 0$ and $\underline u$ is the constant given by
\[
\underline u=\left(\int_{S^n}h\det{H}dS\right)^{-1}\int_{S^n} u h\det(H)dS.
\]

Therefore one gets \begin{prop}[Milman, Proposition 1.4 \cite{Milman22}]
    The local $L^p$-Minkowski Conjecture \ref{locallpminkowski} is equivalent to that the second variation of functional $F_{p,f}$ is nonnegative at critical point $h_K$ for all convex body $K\in \cK_{+,e}^2$, or in short $\lambda_{1, e}\geq n+1-p$.
\end{prop}

Given the variational approach, it is natural to study the structure of critical points of $F$, viewing $(p, f)$ as parameters.
We hope it would shed some lights on the structure of critical points of $F$ in particular when the uniqueness fails (for $p<p_0$) in general. 
A basic reference is B. White \cite{White91}.
Fix $\alpha\in (0, 1)$. 
Consider the space of support functions of even convex bodies  $\cH_e^{2, \alpha}$ and the space of even positive functions $C^\alpha_{e, +}$.
We need to consider maps and functionals on $\cH_e^{2, \alpha}$. For $p\in [0, 1)$, recall the map $F: \R\times C^\alpha_{e, +}\times\cH_e^{2, \alpha}\rightarrow \R,$
\[
\begin{split}
  F(p, f, h)=\frac{1}{p}\log \int_{S^n} h^p d\mathfrak{m}_f-\frac{1}{n+1}\log V(h).
\end{split}
\]
Note that a critical point of $F_{p, f}$ satisfies
\begin{equation}\label{normalized-01}
     h^{1-p}\det(\nabla^2 h+h\text{Id})=(n+1)V(h) \left(\int_{S^n}h^pfdS_x\right)^{-1}f.
\end{equation}
 We recall the set up of \cite[Section 1]{White91}, adapted to our setting.  The first variation  of $F(p, f, h)$ in \eqref{1stvariation-01} defines a map $S: \R\times \cH^{2, \alpha}_e\times C^\alpha_{e, +}\rightarrow C^\alpha_e,$
\begin{equation}
 S(p, f, h)=\left(\int_{S^n}h^p fdS\right)^{-1}h^{p-1}f-\left(\int_{S^n} h\det(\nabla^2 h+h\text{Id})\right)^{-1}\det(\nabla^2 h+h\text{Id}).
\end{equation}
A critical point of $F$ is given by the zeros of $S$, $S(p, f, h)=0$. The linearization of $S$ is given by the Frechet derivative of $S$, with $V=V(h)$, 
\begin{equation}
\begin{split}
     dS_h(p, f, h)(u)=&\;(p-1) \left(\int_{S^n}h^p fdS\right)^{-1}h^{p-2} uf-p\left(\int_{S^n}h^p fdS\right)^{-2}h^{p-1}f \int_{S^n} h^{p-1} uf dS\\
     &\;-\frac{\det(\nabla^2 h+h\text{Id})}{(n+1)V} \left(H^{ij}(u_{ij}+u\delta_{ij})-V^{-1}\int_{S^n} u\det(\nabla^2 h+h\text{Id})dS\right).
\end{split}
\end{equation}
Denote $Ju=dS_{h}(p, f, h)(u)$. It is clear then $J$ is a self-adjoint elliptic operator, hence has Fredholm index zero,
\[
\int_{S^n} Ju\cdot  v\; dS=\int_{S^n} u \cdot Jv\; dS.
\]
Given $(p, f_0)$, let $h_0$ be a critical point of $F$ such that $S(p, f_0, h_0)=0$. Note that $F$ is invariant under the scaling for $h$, namely $F(p, f, h)=F(p, f, ch)$ holds for any positive constant $c$. It follows that $ch_0$ is a critical point of $F$ corresponding to scaling of convex bodies, which also implies $J h_0=0$. To discuss the structure of critical points of $F$, we can then restrict to consider convex bodies with fixed volume by modulo the scaling, $(n+1)V(h)=\omega_n$. Hence we consider the functions which satisfy the constraint
\begin{equation}\label{normalized-03}
\int_{S^n} u \det(\nabla^2 h_0+h_0\text{Id}) dS=0.
\end{equation}
At a critical point $(p, f_0, h_0)$ of $F$, with \eqref{normalized-03}, the ``Jacobi" equation $Ju_0=0$ simplifies to 
\begin{equation}\label{jacobi-01}
    -H^{ij}(u_{ij}+u\delta_{ij})+\frac{(p-1)u}{h}=0.
\end{equation}

We check that the hypotheses (C) in \cite[Theorem 1.2]{White91} holds in our setting.  
 let $u_0\neq 0\in \text{Ker}(J)$ satisfy \eqref{normalized-03} and \eqref{jacobi-01}. 
Let $p(s), f(s)$ be a path such that $p(0)=p, \p_sp=q_0,  f(0)=f_0, \p_sf=w$, we compute
\begin{equation}\label{transverse-01}
\begin{split}
    &\left(\frac{\p^2}{\p s\p t}\right)_{s=t=0} F(p(s), f(s), h_0+tu_0)=\left(\frac{\p}{\p s}\right)_{s=0} \int_{S^n} u_0 S(p(s), f(s), h_0) dS\\
    &\quad=c_0\int_{S^n} u_0 h_0^{p-1}(\log h_0\cdot q_0+w)dS-c_0^2\int_{S^n} u_0 h_0^{p-1}f_0dS \int_{S^n} h_0^p(\log h_0\cdot q_0+w)dS\\
    &\quad=c_0\int_{S^n} u_0 h_0^{p-1}(\log h_0\cdot q_0+w)dS
    \end{split}
\end{equation}
where the constant $c_0$ is given by
\[
c_0=\left(\int_{S^n} h_0^p f_0 dS\right)^{-1}.
\]
Clearly there exist $q_0$ and $w$ such that  \eqref{transverse-01} is not zero. Note that if we take $p$ fixed, then $q_0=0$, and there exists $f$ such that \eqref{transverse-01} is not zero as well.  We summarize the results as follows, as a direct application of \cite[Theorem 1.2]{White91}.

\begin{thm}Let $h_0$ be a critical point of $F(p_0, f_0, \cdot)$ such that $S(p_0, f_0, h_0)=0$. Then the map $S$ is a submersion in a neighborhood of $(p_0, f_0, h_0)$ so that there exists a neighborhood $W$ of $(p_0, f_0, h_0)$ such that
\[
\cM=\{(p, f, h)\in W: S(p, f, h)=0\}
\]
is a submanifold of $\R\times  C^{\alpha}_{+, e}\times \cH^{2, \alpha}_{+, e}$ and
\[
T_{(p, f, u)}\cM=\text{Ker} (DS).
\]
The projection $\Pi: \cM\rightarrow \R\times  C^{\alpha}_{e}$ is a Fredholm map of Fredholm index zero. 
\end{thm}

When $p$ is fixed, only $f$ is viewed as a parameter, we have a similar description and  the statement holds by dropping the dependence on $p$.

\section{The proof of the main theorem}
We prove our main theorem in this section. 
First we prove the following.

\begin{thm}\label{uniqueness-01}
    Suppose $p\in [0, 1)$ such that the even $L_p$ Minkowski inequality holds. Then at $p$ the even $L_p$ Minkowski problem has a unique solution when $f$ is a positive even function with uniform $C^\alpha$ bound. 
\end{thm}

\begin{proof}First note that if the $L_p$ Minkowski inequality holds, then a critical point of $F_{p, f}$ is necessarily an energy minimizer. 
Suppose for some $f$, $F_{p, f}$ has more than one critical point which solves the even $L_p$ Minskowski problem. 
Let $h_L$ and $h_K$ both minimize $F_{p, f}$. Then both $h_L$ and $h_K$ satisfy the Euler-Lagrangian equation (after a normalization condition)
\begin{equation}\label{normal-001}
\left({\displaystyle\int_{S^{n}}h^{p} fdS}\right)^{-1}h^{p-1} f=\frac{\det(\nabla^2 h+h \text{Id})}{(n+1)V(h)}.    
\end{equation}
We can assume $V(h_L)=V(h_K)$. 
By
$F_{p, f}(h_L)=F_{p, f}(h_K)$, it
 implies that 
 \begin{equation}
     \frac{1}{p}\int_{S^{n-1}} h_L^{p} fdS=\frac{1}{p}\int_{S^{n-1}} h_K^{p} fdS,
 \end{equation} 
 where $p=0$ is understood as the limit sense. Hence it follows that $h_L$ and $h_K$ solve the equation
 \begin{equation}
     h^{1-p}\det(\nabla^2 h+h\text{Id})=cf,
 \end{equation}
 where $c$ is the constant taking the form
 \[
 c=(n+1)V(h_L) \left(\int_{S^n} h_L^p fdS\right)^{-1}=(n+1)V(h_K) \left(\int_{S^n} h_K^p fdS\right)^{-1}. 
 \]
 Now denote
\[
q_\l=\left((1-\l) h_L^{p}+\l h_K^{p}\right)^{\frac{1}{p}}.
\]
Recall  $(1-\l)L+_{p}\l K$ is defined by the Alexanderov body associated with $q_\l$,
\[
Q_\l=\{x\in \R^n: x\cdot u\leq q_\l(u)\}. 
\]
Denote the support function of $Q_\l$ by $h_\l$, then $h_\l\leq q_\l$ and $h_\l=q_\l$ a.e with respect to the surface area measure $dS_{Q_\l}$.
Now we compute the volume of $V(Q_\l)$ as in \cite[Lemma 3.1]{BLYZ01},
\begin{equation}
\begin{split}
     V(Q_\l)=&\;\frac{1}{n}\int_{S^{n-1}} h_\l dS_{Q_\l}\\
     =&\;\frac{1}{n}\int_{S^{n-1}} q_\l^p h_\l^{1-p} dS_{Q_\l}\\
     =&\;\frac{1}{n}\int_{S^{n-1}} ((1-\l) h_L^p+\l h_K^p) h_\l^{1-p} dS_{Q_\l}\\
     \geq &\; (1-\l)V(Q_\l)^{\frac{n-p}{n}} V(L)^{\frac{p}{n}}+\l V(Q_\l)^{\frac{n-p}{n}} V(K)^{\frac{p}{n}},
\end{split}
\end{equation}
where in the last step of the above, we apply the $L_p$-Brunn-Minkowski inequality \eqref{BM-02}. 
It follows that $V(Q_\l)\geq V(K)=V(L)$. This implies 
\[
F_{{p}, f}(h_\l)\leq F_{p, f}(h_L)=F_{p, f}(h_K).
\]
Hence $h_\l$ minimizes $F_{p, f}(h)$ and $V(Q_\l)=V(K)=V(L)$ for all $\l\in [0, 1]$ and for $\lambda\in [0, 1]$
\[
\int_{S^n} h_\lambda^p fdS=\int_{S^n} h_L^pfdS.
\]
It follows that \cite[Lemma 4.1]{BLYZ02} $h_\lambda$ solves the equation 
\[
h^{1-p}\det(\nabla^2 h+h I)=cf.
\]
Moreover, $C^{-1}\leq h_\lambda\leq C$, see for example \cite[Lemma 2.1]{CHLL}. 
The regularity theory of Caffarelli \cite{Caff89, Caff90} then implies that $h_\lambda$ is $C^{2, \alpha}$ for each $\lambda\in [0, 1]$ and $dS_{h_\lambda}$ is then given by a $C^\alpha$ potential.  The continuity then implies that  $h_\lambda=q_\lambda$.  
When $p=0$, $h_\lambda=h_K^{1-\lambda}h_L^\lambda$, hence we have
\[
\p_\lambda h= h w, w=\log h_L-\log h_K.  
\]
By taking derivative of 
\[
\log h+\log \det(\nabla^2 h+hI)=\log f,
\] we have
\begin{equation}
    w+H^{ij}( (hw)_{ij}+hw \delta_{ij})=0.
\end{equation}
That is,
\begin{equation}\label{linear-001}
    (n+1)w+H^{ij}(h w_{ij}+h_iw_j+h_jw_i)=0.
\end{equation}
Taking derivative with respect to $\lambda$ one more time, we get
\begin{equation}
    -H^{ik}H^{jl}( (hw)_{ij}+hw \delta_{ij})( (hw)_{kl}+hw \delta_{kl})+ H^{ij}((hw^2)_{ij}+hw^2\delta_{ij})=0.
\end{equation}
Denote $a_{ij}=hw_{ij}+h_iw_j+h_jw_i$ and we have $H^{ij}a_{ij}=-(n+1)w$.
We compute
\begin{equation}
\begin{split}
     H^{ik}H^{jl}( (hw)_{ij}+hw \delta_{ij})( (hw)_{kl}+hw \delta_{kl})=&H^{ik}H^{jl}(H_{ij}w+a_{ij})(H_{kl}w+a_{kl})\\
     =&H^{ik}H^{jl}a_{ij}a_{kl}-(n+2)w^2.
\end{split}
\end{equation}
We also compute
\begin{equation}
     H^{ij}((hw^2)_{ij}+hw^2\delta_{ij})=H^{ij}(w^2 H_{ij}+2w a_{ij}+2hw_iw_{j})=-(n+2)w^2+2H^{ij}hw_iw_j
\end{equation}
This gives
\begin{equation}
    H^{ij} hw_iw_j=\frac{1}{2}H^{ik}H^{jl}a_{ij}a_{kl}.
\end{equation}
At the maximum of $w$, we have $\nabla w=0$.  It follows that, $a_{ij}=0$. This implies $\nabla^2 w=0$. Since $w$ satisfies
\[
H^{ij}(h w_{ij}+h_iw_j+h_jw_i)=-(n+1)w,
\]
this implies $\max w=0$ and hence $w=0$. 
Hence we know that, given $f$, the uniqueness holds for $p=0$. 
For general $p\in (0,1)$, the argument is essentially the same. Let $h_\lambda=[(1-\lambda)h_L^p+\lambda h_K^p]^{\frac{1}{p}}$. Then \[
\partial_\lambda h_\lambda=h_\lambda^{1-p}w,\; \partial_\lambda^2h_\lambda=(1-p)h^{1-2p}_{\lambda}w^2, \quad \text{ where } w=\frac{h_K^p-h_L^p}{p}.
\]Set $A_{ij}=(\partial_\lambda h_\lambda)_{ij}+\partial_\lambda h_\lambda \delta_{ij}$. We compute \[\begin{aligned}
A_{ij}=&(h_\lambda^{1-p}w)_{ij}+h^{1-p}_\lambda w\delta_{ij}\\
=&h_{ij}(h^{-p}w)+h_{i}(h^{-p}w)_j+h_{j}(h^{-p}w)_i+h(h^{-p}w)_{ij}+h^{-p}w\cdot hw\delta_{ij}\\
=&h^{-p}wH_{ij}+h_{i}(h^{-p}w)_j+h_{j}(h^{-p}w)_i+h(h^{-p}w)_{ij}.
\end{aligned}
\]
Take derivative with respect to $\lambda$ for the equation we get \begin{equation}
    (1-p)h^{-p}w+H^{ij}A_{ij}=0.
\end{equation}
So \begin{equation}
    \begin{aligned}
        &H^{ij}(\partial^2_\lambda h_{\lambda ij}+\partial^2_\lambda h_\lambda \delta_{ij})\\
        =&(1-p)H^{ij}\left\{(h\cdot h^{-2p}w^2)_{ij} + h\cdot h^{-2p}w^2\delta_{ij}\right\}   \\
        =&(1-p)H^{ij}\left\{2h\cdot h^{-p}w(h^{-p}w)_{ij}+2h_i(h^{-p}w)_j(h^{-p}w)+2h_j(h^{-p}w)_ih^{-p}w\right\} \\
        &+(1-p)H^{ij}\left\{h_{ij}h^{-2p}w^2+h^{-2p}w^2h\delta_{ij}\right\}+2(1-p)H^{ij}h(h^{-p}w)_{i}(h^{-p}w)_j\\
        =&2(1-p)H^{ij}(A_{ij}h^{-p}w-h^{-p}wH_{ij})+n(1-p)h^{-2p}w^2+2(1-p)H^{ij}h(h^{-p}w)_{i}(h^{-p}w)_j\\
        =&-2(1-p)^2h^{-2p}w^2-2n(1-p)h^{-2p}w^2+2(1-p)H^{ij}h(h^{-p}w)_{i}(h^{-p}w)_j\\
        =&2(p-1)(n+1-p)h^{-2p}w^2+2(1-p)H^{ij}h(h^{-p}w)_i(h^{-p}w)_j
    \end{aligned}
\end{equation}
Therefore taking derivative of the equation again gives \begin{equation}
\begin{aligned}
     0=&-H^{ik}H^{lj}A_{ij}A_{lk}+H^{ij}(\partial^2_\lambda h_{\lambda ij}+\partial^2_\lambda h_\lambda \delta_{ij})\\
     =&-H^{ik}H^{lj}A_{ij}A_{lk}+2(1-p)H^{ij}h(h^{-p}w)_i(h^{-p}w)_j-2(1-p)(n+1-p)h^{-2p}w^2\end{aligned}
    \end{equation}
    It follows that $w=0$ by maximum principle, since at the critical points of $h^{-p}w$, one has \[H^{ik}H^{lj}A_{ij}A_{lk}=-2(1-p)(n+1-p)h^{-2p}w^2\leq 0.\]
\end{proof}
We recall the definition of $p_0$, 
\[
\inf_{K\in \cK^2_{+, e}} \lambda_{1, e}(-L_K)=n+1-p_0.
\]
The result in \cite{KM01} shows that $0\leq p_0<1$. Combining the results in \cite{KM01} and \cite{CHLL}, for any $p\in (p_0, 1)$ and for each positive even smooth function $f$, the even $L_p$ Minkowski problem has a unique solution, 
\[
h^{1-p}\det(\nabla^2 h+h\text{Id})=f.
\]
It implies the $L_p$ Brunn-Minkowski inequality holds for $[p_0, 1)$ in particular.
Note that we can also use the result by Putterman \cite[Propostion 1.5]{Put}, that the $L_p$ Brunn-Minkowski inequlaity holds at $p_0$. 
Hence by Theorem \ref{uniqueness-01} the uniqueness holds at $p_0$ .
Next we prove that $\lambda_{1, e}(-L_K)>n+1-p_0$ if the support function $K\in \cK^4_{+, e}$. 

\begin{proof}
Suppose $K$ takes a local minimum eigenvalue such that  $\lambda_{1, e}(-L_K)=n+1-p_0$. Consider $K_t$ to be a path of convex bodies with $K_0=K$ and $\lambda_{1, e}(L_{K_t})\geq n+1-p_0$. 
Let $z$ be an eigenfunction of $n+1-p_0$ with respect to $-L_{K}$, satisfying
\begin{equation}
    H^{ij}(hz_{ij}+h_iz_j+h_jz_i)=-(n+1-p_0)z.
\end{equation}
Then we define
\begin{equation}
g(t)=\int_{S^n} H^{ij}_t h^2_tz_iz_j \det(H_t) dS-(n+1-p_0) \int_{S^n} h_t z^2 \det{H_t} dS+(n+1-p_0)s(t),
\end{equation}
where $s(t)$ takes the form
\[
s(t)=\left(\int_{S^n}h_t \det(H_t)dS\right)^{-1} \left(\int_{S^n} h_t z \det(H_t) dS\right)^2.
\]
Note that $s(0)=0$, and $s^{'}(0)=0$.
Clearly we have $g(0)=0$ and $g(t)\geq 0$. Hence we have $g^{'}(0)=0$.
Note that we also have $g^{''}(0)\geq 0$. 
We write 
\[
g_1(t)=\int_{S^n} H^{ij}_t h^2_tz_iz_j \det(H_t) dS,\; g_2(t)=\int_{S^n} h_t z^2 \det{H_t} dS.
\]
We first compute, assuming $\p_t h=w$ at $t=0$ for an even function  $w$,
\begin{equation}
\begin{split}
     g_2^{'}(0)=&\int_{S^n} wz^2 \det{H} dS+\int_{S^n} hz^2 H^{kl}(w_{kl}+w\delta_{kl})\det{H} dS\\
     =&\int_{S^n} wz^2\det{H}dS+\int_{S^n} wz^2 H^{kl} h\delta_{kl}dS+\int_{S^n} wH^{kl}(hz^2)_{kl}\det{H}dS
\end{split}
\end{equation}
We compute
\[
\begin{split}
H^{kl}(hz^2)_{kl}=&\;H^{kl}(z^2 h_{kl}+2z(hz_{kl}+h_{k}z_l+h_lz_k)+2hz_kz_l)\\
=&\;z^2 H^{kl} h_{kl}-2(n+1-p_0)z^2+2hH^{kl}z_kz_l.
\end{split}
\]
Hence we have
\begin{equation}
    g^{'}_2(0)=-(n+1-2p_0)\int_{S^n} wz^2\det{H}dS+\int_{S^n} 2w H^{kl}hz_kz_l \det{H} dS.
\end{equation}
The computation for $g_1(t)$ is a bit involved,
\begin{equation}
\begin{split}
    g_1^{'}(0)=&\;\int_{S^n}(H^{ij}H^{kl}-H^{ik}H^{jl})(w_{kl}+w\delta_{kl})h^2z_iz_j\det{H}+\int_{S^n}2w H^{ij}hz_iz_j\det{H}\\
    =&\;\int_{S^n} w \left[(H^{ij}H^{kl}-H^{ik}H^{jl})h^2 z_iz_j\det{H}\right]_{kl}+\int_{S^n}2w H^{ij}hz_iz_j\det{H}\\&\;+\int_{S^n} w (H^{ij}H^{kl}-H^{ik}H^{jl})\delta_{kl}h^2z_iz_j\det{H}.
    \end{split}
\end{equation}
It follows that
\begin{equation}
\begin{split}
      g^{'}(0)=&\;\int_{S^n} w \left[(H^{ij}H^{kl}-H^{ik}H^{jl})h^2 z_iz_j\det{H}\right]_{kl}+\int_{S^n} w (H^{ij}H^{kl}-H^{ik}H^{jl})\delta_{kl}h^2z_iz_j\det{H}\\
      &\;-\int_{S^{n}} 2(n-p_0) w H^{ij}hz_iz_j\det{H}+(n+1-p_0)(n+1-2p_0)\int_{S^n} wz^2 \det{H}.
\end{split}
\end{equation}
Hence we derive the Euler-Lagrangian equation 
\begin{equation}
\begin{split}
    &\left[(H^{ij}H^{kl}-H^{ik}H^{jl})h^2 z_iz_j\det{H}\right]_{kl}+(H^{ij}H^{kl}-H^{ik}H^{jl})\delta_{kl}h^2z_iz_j\det{H}\\&\;-2(n-p_0)  H^{ij}hz_iz_j\det{H}
    +(n+1-p_0)(n+1-2p_0) z^2\det{H}=0.
    \end{split}
\end{equation}
Now we consider at a maximum point of $z$, we have $z_i=0$, and hence the equation reduces to 
\begin{equation}
    \left[(H^{ij}H^{kl}-H^{ik}H^{jl})h^2 z_iz_j\det{H}\right]_{kl}+(n+1-p_0)(n+1-2p_0) z^2\det{H}=0.
\end{equation}
Since $(z_iz_j)_k=z_{ik}z_j+z_{i}z_{jk}=0$ at the point, hence the equation reduces further to
\begin{equation}
    (H^{ij}H^{kl}-H^{ik}H^{jl})h^2 (z_iz_j)_{kl}+(n+1-p_0)(n+1-2p_0) z^2=0.
\end{equation}
Now we compute, at the point, 
\[
(z_iz_j)_{kl}=z_{ikl}z_j+z_{ik}z_{jl}+z_{il}z_{jk}+z_iz_{jkl}=z_{ik}z_{jl}+z_{il}z_{jk}.
\]
Hence we get 
\[
(H^{ij}H^{kl}-H^{ik}H^{jl})h^2 (z_{ik}z_{jl}+z_{il}z_{jk})+(n+1-p_0)(n+1-2p_0) z^2=0.
\]
We compute, at the point, 
\[
H^{ik}z_{ik}h=-(n+1-p_0)z.
\]
Hence we get
\[
H^{ij}H^{kl}h^2 z_{ik}z_{jl}=p_0(n+1-p_0)z^2.
\]
However, we have
\[
H^{ij}H^{kl}h^2 z_{ik}z_{jl}\geq \frac{(H^{ij}hz_{ij})^2}{n}=\frac{(n+1-p_0)^2z^2}{n}.
\]
This gives that
\[
(n+1)(1-p_0)z^2\leq 0.
\]
 For $p_0<1$, this gives $\max z=0$ and hence $z=0$. Contradiction. This proves that 
 \[
 \lambda_{1, e}(-L_K)>n+1-p_0.
 \]
 \end{proof}

Our result gives a characterization of the uniqueness of even $L_p$ Minkowski problem for $p\in [0, 1)$, in terms of $\inf \lambda_{1, e}$. 
The conjecture \cite[Section 5]{KM01} asserts that $\inf \lambda_{1, e}=n+1$ ($p_0=0$). 
There are many evidences for the conjecture to hold in general, see \cite{KM01, Milman23, B24} for more discussions. Note that the conjecture holds for $n=1$ \cite{BLYZ01} and all known non-uniqueness results require $p<0$.   The $L_\infty$ estimate in \cite{CHLL} provides another evidence for $p_0=0$, which only works for $p\geq 0$ and fails for $p<0$ \cite{JLW}. We include a modification of their $L_\infty$ estimate below. 
\begin{lemma}\cite[Lemma 2.1]{CHLL}
    Let $p\in [0, 1]$. Suppose $K\in \cK_e$ satisfies \eqref{normalized-01} with $(n+1)V(K)=\omega_n$. Suppose there exists a positive constant $C_0$ such that
    $C_0^{-1}\leq f\leq C_0$. Then there exists a positive constant $C_1$ depending only $C_0$ and $n$ such that
        $C_1^{-1}\leq h_K\leq C_1$.     
\end{lemma}
\begin{proof}
 If $h_K$ solves \eqref{normalized-01} with $(n+1)V(K)=\omega_n$, then $a h_K$ solves  the equation
    \[
    h^{1-p}\det(\nabla^2 h+h\text{Id})=f,
    \]
    for a positive constant $a$ satisfying 
    \begin{equation}\label{normalized-02}
        \int_{S^n} h_K^p fdS_x=a^{n+1-p} \omega_n.
    \end{equation}
    By \cite[Lemma 2.1]{CHLL}, there exists a positive constant $C$ depending only on $C_0, n$ such that $C^{-1}\leq a h_K\leq C$. Together with \eqref{normalized-02}, $a$ is then a uniformly bounded constant. This implies that $h_K$ satisfies the desired estimates. 
\end{proof}

\section{The first even eigenvalue}

First we discuss a totally bound Radon measure associated with a convex body $K\in \cK_e$. Given a function $u\in C^2(S^n)$, we consider the following measure 
defined by $K$ and $u$,
\[
\mathfrak{m}_{K, u}=U^{ij}u_iu_j dS,
\]
where $U^{ij}$ is the cofactor of $H_{ij}=h_{ij}+h\delta_{ij}$. When $H$ is invertible, $U^{ij}=H^{ij}\det{H}$. We have the following, 

\begin{thm}Fix a function $u\in C^2(S^n).$ Suppose $K_i\in\cK^2_{+, e}$ converges to $K$ in the Hausdorff topology, then $\mathfrak{m}_{K_i, u}$ converges weakly to $\mathfrak{m}_{K, u}$. 
\end{thm}

\begin{proof}

Choose $e_1,e_2,...,e_n,e_{n+1}$ to be an orthonormal frame of $\mathbb{R}^{n+1}$ such that $e_1,e_2,...,e_n$ is tangent to $S^n$ and $e_{n+1}$ is the outer unit normal when restricted to $S^n$. We use $D$ and $\nabla$ to denote the covariant derivatives in $\mathbb{R}^{n+1}$ and $S^n$ respectively. Then, for $1\leq i,j\leq n$, $h$ a homogeneous function of degree $1$, \[
D^2_{e_i,e_j}h=D_{e_i}D_{e_j}h-D_{D_{e_i}e_j}h=\nabla_{e_i}\nabla_{e_j}h-D_{\nabla_{e_i}e_j+\langle D_{e_i}e_j,e_{n+1}\rangle e_{n+1}}h=\nabla^2_{e_i,e_j}h+h\delta_{ij}.
\]In the above we have used that the second fundamental form of the sphere $S^n$ is $\langle D_{e_i}e_j,e_{n+1}\rangle=-\delta_{ij},$ and $h$ satisfies $D_{e_{n+1}}h=h$ on $S^n$ since $h$ is homogeneous of degree $1$. Moreover, if we take twice derivative with respect to $t$ at $t=1$ in the equation 
$h(tx)=th(x)$, we get $D^2_{e_{n+1},e_{n+1}}h(x)=0$ for $x\in S^n$. For the mixed derivative, \[
D^2_{e_i,e_{n+1}}h=D_{e_i}D_{e_{n+1}}h-D_{D_{e_i}e_{n+1}}h=D_{e_i}h-\sum_jD_{\langle D_{e_i}e_{n+1},e_j\rangle e_j}h=D_{e_i}h-\sum_j\delta_{ij}D_{e_j}h=0.
\]
So the Hessian of $h$ is given by \[
D^2h(tx)=\begin{pmatrix}
    \frac{1}{t}(\nabla^2 h+h\text{Id})(x)&0\\
    0&0
\end{pmatrix}\]
Let $r=|x|$, then we consider the symmetric $2$-form $A:=D^2h+dr\otimes dr.$ Embed $\mathbb{R}^{n+1}$ into $\mathbb{R}^{n+1}\times \mathbb{T}^{n+1}$, extend $h,r$ constantly on the $\mathbb{T}^{n+1}$ fibers. Let $x$ be the coordinate for $\mathbb{R}^{n+1}$ and $y$ for $\mathbb{T}^{n+1}$. With the natural complex structure $z_\alpha=x_\alpha+\sqrt{-1}y_\alpha$, we compute \[\partial_\alpha\bar{\partial}_\beta h=\frac{1}{4}D^2h(\partial_{x_\alpha},\partial_{x_\beta}).\] Let $B$ be the $(1, 1)$ form defined by  \[B=\sqrt{-1}B_{\alpha\bar{\beta}}dz_\alpha\wedge d\bar{z}_\beta:=\sqrt{-1}\partial\bar{\partial}h+\sqrt{-1}\partial r\wedge\bar{\partial}r.\] We get $4B_{\alpha\bar{\beta}}=A_{\alpha\beta}$. Now for a smooth function $u$ on $S^n$, we extend it to be a homogeneous function of degree $0$ to $\mathbb{R}^{n+1}$ first and further constantly on the $\mathbb{T}^{n+1}$ fibers. Then on $S^n$, we have \[
H^{ij}u_iu_j\det(H)=A^{\alpha\beta}u_\alpha u_\beta \det A= 4^{n}B^{\alpha\bar{\beta}}u_\alpha u_{\bar{\beta}}\det B.
\]
We write the standard volume form on $\R^{n+1}\times \mathbb{T}^{n+1}$ as
\[
dv=\left(\frac{\sqrt{-1}}{2}\right)^{n+1}dz_1\wedge d\bar{z}_1\wedge\cdots\wedge dz_{n+1}\wedge d\bar{z}_{n+1}.
\]
Then we have 
\[(n+1)\sqrt{-1}\partial u\wedge\bar{\partial}u\wedge B^n=2^{n+1}(n+1)!B^{\alpha\bar{\beta}}u_\alpha u_{\bar{\beta}}\det B dv.\]
Since $dx=r^{n}drdS$, where $dS$ is the area measure of $S^n$, we have  \[dv=dx\wedge dy=r^{n}drdSdy.\] 
By the homogeneity, we know $(A^{\alpha\beta}u_\alpha u_\beta\det A)(rx)=r^{-n-1}(A^{\alpha\beta}u_\alpha u_\beta\det A)(x).$ Therefore, \[
r\sqrt{-1}\partial u\wedge\bar{\partial}u\wedge B^n(rx)/(dr\wedge dS^n\wedge dy) \text{ is independent of }r.
\]
It follows that, for a constant depending only on $n$, 
\[
\int_{S^n}H^{ij}u_iu_j\det HdS=c_n\int_0^2\int_{\mathbb{T}^{n+1}}\int_{S^n}\chi(r)r\sqrt{-1}\partial u\wedge\bar{\partial} u\wedge B^n(rx),
\] where $\chi(r)$ is a nonnegative smooth function supported in $[\frac{1}{2},2]$ satisfying  \[\int_0^2\chi(r)dr=1.\]
Note that we  have $B^n=(\ddc h)^n+n(\ddc h)^{n-1}\wedge\sqrt{-1}\partial r\wedge\bar{\partial}r$. Since $\sqrt{-1}\partial \bar\partial h(\p_r, \cdot)=0$, we have, for a function $u$ independent of $r$, 
\[
\sqrt{-1}\partial u\wedge\bar{\partial} u\wedge B^n=n(\ddc h)^{n-1}\wedge\sqrt{-1}\partial r\wedge\bar{\partial}r\wedge \sqrt{-1}\partial u\wedge \bar\partial u.
\]
This gives the following, for $M=\R^{n+1}\times \mathbb{T}^{n+1}$, 
\begin{equation}\label{identity-001}
    \int_{S^n}H^{ij}u_iu_j\det HdS=c(n)\int_{M} \chi{(r)} r (\sqrt{-1}\partial r\wedge \bar\partial r)\wedge (\sqrt{-1}\partial u\wedge \bar\partial u)\wedge (\sqrt{-1}\partial\bar\partial h)^{n-1}.
\end{equation}
Note that the constant $c(n)$ does not affect our results and our argument, and we can choose a normalization such that the torus has a fixed volume, making the the constant $c(n)=1$. 

\begin{lemma}[Chern–Levine–Nirenberg inequality]
    Let $\beta$ be the standard positive $(1,1)$-form $\beta=\sum_{i=1}^{n+1}\frac{1}{2}\sqrt{-1}dz^i\wedge d\bar{z}^i$ on $\mathbb{R}^{n+1}\times \mathbb{T}^{n+1}$, and $h_1,..., h_k\in C^2$ such that $\ddc h_i\geq 0$. For any compact set $K$, there exists a constant $C$ depending on $K,n,k$ such that  \[
    \int_{K}\beta^{n+1-k}\wedge(\ddc h_1)\wedge\cdots \wedge(\ddc h_k)\leq C||h_1||_{L^\infty}\cdots ||h_k||_{L^\infty}
    \]
\end{lemma}
Now given two $C^2$ support functions $h_1, h_2$,  we can estimate the difference of two such integrals as follows.
Denote, for any smooth test function $\varphi\in C^\infty(S^n)$, 
\begin{equation}
    I:=\left|\int_{S^n}\varphi H_1^{ij}u_iu_j\det H_1dS^n-\int_{S^n}\varphi H_2^{ij}u_iu_j\det H_2dS^n\right|.
\end{equation}
We compute, using \eqref{identity-001} with $c(n)=1$, 
\[
\begin{split}
   I=&\left|\int_{M}\chi(r)r\varphi\sqrt{-1}\partial u\wedge\bar{\partial} u\wedge\sqrt{-1}\partial r\wedge\bar{\partial}r\wedge((\ddc h_1)^{n-1}-(\ddc h_2)^{n-1})\right|\\
   =&\left|\int_{M}\chi(r)r\varphi\sqrt{-1}\partial u\wedge\bar{\partial} u\wedge\sqrt{-1}\partial r\wedge\bar{\partial}r\wedge (\sqrt{-1}\partial\bar\partial(h_1-h_2))\wedge \Theta\right|\\
   =&\left|\int_M (h_1-h_2) \ddc(\chi(r)r\varphi\sqrt{-1}\partial u\wedge\bar{\partial} u\wedge\sqrt{-1}\partial r\wedge\bar{\partial}r)\wedge \Theta \right|\\
   \leq &C\|h_1-h_2\|_{L^\infty}\int_M\beta^3\wedge \Theta ,
\end{split}
\]where we write \[(\ddc h_1)^k-(\ddc h_2)^k=\ddc(h_1-h_2)\wedge \Theta, \; \Theta=\sum_{i=1}^{k-1}(\ddc h_1)^{i}(\ddc h_2)^{k-1-i},\]
and \[
-C\beta^3\leq \ddc(\chi(r)\varphi\sqrt{-1}\partial u\wedge\bar{\partial} u\wedge\sqrt{-1}\partial r\wedge\bar{\partial}r)\leq C\beta^3,
\]for some constant $C$ depending on $\chi, ||\varphi||_{C^2},||u||_{C^2}.$

We get the above difference $I$ is bounded by, from the Chern–Levine–Nirenberg inequality, \[
I\leq C(\chi,||\varphi||_{C^2}, ||u||_{C^2}, ||h_1||_{L^\infty},||h_2||_{L^\infty},n)||h_1-h_2||_{L^\infty}.
\]
This proves that the weak convergence of the measures $H^{ij}u_iu_j\det HdS^n$ when the support functions converge in $L^\infty$. So for any convex body with support $h\in C(S^n)$, one can define the measure $H^{ij}u_iu_j\det HdS$ as the weak limit of any smooth approximations $H_k^{ij}u_iu_j\det H_kdS^n$ for $h_k\in C^\infty$ and $h_k\rightarrow h$ in $L^\infty$. And by the Chern-Levine-Nirenberg inequality, the total measure $||H^{ij}u_iu_j\det HdS||$ is uniformly bounded by constant depending on $||u||_{C^2},||h||_{L^\infty},n.$ 
\end{proof}

\begin{Remark}
    Using the complex notations and plurisubharmonic functions to define the measures is only technical and auxiliary. The algebraic structure behind the operator is the well-known mixed volume in convex geometry  and it corresponds to, essentially, the wedge product of corresponding matrices. The choice of complex notations is to use wedge product of forms, which is more transparent than the wedge product of matrices.  
\end{Remark}

Similarly, $U^{ij}\delta_{ij}dS$ is also a well-defined Radon measure for a convex body $K\in \cK_e$. 
We  now recall two equivalent  definitions of  the first even eigenvalue of $-L_K$ for any $K\in \cK_e$ as follows, when the convex body is in $\cK^{2}_{+, e}$. The first is the eigenvalue of the operator $-L_K$, 
\begin{equation}
    \lambda_{1, e}(-L_K)=\inf_{u\in C^2_e(S^n)} \frac{\displaystyle \int_{S^n}H^{ij}h^2u_iu_j\det(H)dS}{\displaystyle\int_{S^n}{(u-\underline u)^2} h\det(H)dS}.
\end{equation}
for $u-\underline u\neq 0$. The second is essentially taking $z=hu$ in the above, and define
\begin{equation}\label{eigenvalue-003}
    \lambda_{1, e}(K)=\inf_{z\in C^2_e(S^n)}\frac{\displaystyle \int_{S^n}H^{ij}z_iz_j\det(H)dS-\int_{S^n}H^{ij}\delta_{ij}z^2 \det{H}dS+n\int_{S^n}z^2h^{-1}\det{H}dS}{\displaystyle\int_{S^n}z^2h^{-1}\det(H)dS-\left(\int_{S^n}h\det{H}dS\right)^{-1}\left(\int_{S^n}z\det{H}dS\right)^2}.
\end{equation}
By our results above, both definitions extend to nonsmooth even convex bodies. Next, we prove that the two formulations are equivalent. 
The following is straightforward.
\begin{prop}
    Suppose $K_k\in \cK^2_{+, e}$ converges to $K_i$ in the Hausdorff topology, then $\lambda_{1, e}(-L_K)\geq \limsup \lambda_{1, e}(-L_{K_i})$. The same holds true for $\lambda_{1, e}(K)$.
\end{prop}

\begin{proof}
By the definition of $\lambda_{1, e}(-L_K)$, it is not clear at all whether there exists an eigenfunction. Fix an $\epsilon>0$. Choose $u\in C^2_e(S^n)$ such that 
\[
\lambda_{1, e}(-L_K)+\epsilon= \int_{S^n}H^{ij}h^2u_iu_j\det(H)dS, \int_{S^n}{(u-\underline u)^2} h\det(H)dS=1.
\]
Fix the smooth function $u\in C^2(S^n)$. First we choose a smooth function $\tilde{h}$, such that $||\tilde{h}-h||_{L^\infty}<\epsilon$. 
For such a sequence $h_k\rightarrow h$ in $L^\infty$, we estimate \[
\begin{aligned}
    II:=&\;|\int_{S^n}(h_k^2H_k^{ij}u_iu_j\det H_k-h^2H^{ij}u_iu_j\det H)dS^n|\\
    \leq &\;\int_{S^n}|h_k^2-\tilde{h}^2|H_k^{ij}u_iu_j\det H_kdS^n+\int_S^n|\tilde{h}^2-h^2|H^{ij}u_iu_j\det HdS^n\\
    &+|\int_{S^n}\tilde{h}^2(H_k^{ij}u_iu_j\det H_k-H^{ij}u_iu_j\det H)dS^n|\\
    \leq &\;C(||h_k-h||_{L^\infty}+\epsilon)||H_k^{ij}u_iu_j\det H_k||+\epsilon||H^{ij}u_iu_j\det HdS^n||\\
    &+|\int_{S^n}\tilde{h}^2(H_k^{ij}u_iu_j\det H_k-H^{ij}u_iu_j\det H)dS^n|.
\end{aligned}
\]Let $k\rightarrow \infty$, one gets \[
II\leq C\epsilon.
\]Since $\epsilon>0 $ is arbitrary, we get that the integral $\int_{S^n}h^2H^{ij}u_iu_j\det HdS^n$ is equal to the limit of $\int_{S^n}h_k^2H_k^{ij}u_iu_j\det H_kdS^n$ when $h_k\rightarrow h$ in $L^\infty$. Moreover, it is straightforward to see that 
\[
\int_{S^n}{(u-\underline u)^2} h_k\det(H_k)dS\rightarrow \int_{S^n}{(u-\underline u)^2} h\det(H)dS.
\]
By definition, we have
\[
\lambda_{1, e}(-L_{K_k})\leq \left(\int_{S^n}{(u-\underline u)^2} h_k\det(H_k)dS\right)^{-1}\int_{S^n}H^{ij}_kh^2_ku_iu_j\det(H_k)dS
\]
Hence it follows that, for a uniformly bounded constant $C$, 
\[
\lambda_{1, e}(-L_K)+\epsilon \geq (1-C\epsilon)\lambda_{1, e}(-L_{K_k})-C\epsilon.
\]
The proof for $\lambda_{1, e}(K)$ is essentially the same. 
This completes the proof. 
\end{proof}

\begin{prop}
    For a convex body $K\in \cK_e$ with the support function $C^{-1}\leq h\leq C$, then $\lambda_{1, e}(-L_K)=\lambda_{1, e}(K)$. 
\end{prop}
\begin{proof}
Fix $\epsilon>0$.
    Suppose $\lambda_{1, e}(-L_K)+\epsilon$ is obtained by a smooth even function $u$, such that 
    \[
    \int_{S^n}{(u-\underline u)^2} h\det(H)dS=1.
    \]
    Now fix $u$. Choose $\tilde h$ smooth, $\|\tilde h-h\|<\epsilon_{1}(\epsilon, u)$  and denote
    \[
    \tilde \lambda=\frac{\displaystyle \int_{S^n}\tilde H^{ij}\tilde h^2u_iu_j\det(\tilde H)dS}{\displaystyle\int_{S^n}{(u-\underline u)^2} \tilde h\det(\tilde H)dS}.
    \]
    We can choose $\epsilon_1$ sufficiently small such that 
    \[
   -C\epsilon\leq \lambda_{1, e}(-L_K)-\tilde \lambda \leq C\epsilon.
    \]
    Next take $\tilde{z}=\tilde h u$, then we get 
    \[
    \tilde \lambda=\frac{\displaystyle \int_{S^n}\tilde H^{ij}\tilde{z}_i\tilde{z}_j\det(\tilde H)dS-\int_{S^n}\tilde H^{ij}\delta_{ij}\tilde{z}^2 \det{\tilde H}dS+n\int_{S^n}\tilde{z}^2\tilde h^{-1}\det{\tilde H}dS}{\displaystyle\int_{S^n}\tilde{z}^2\tilde h^{-1}\det(\tilde H)dS-\left(\int_{S^n}\tilde h\det{\tilde H}dS\right)^{-1}\left(\int_{S^n}\tilde{z}\det{\tilde H}dS\right)^2}.
    \] 
Except the term $\displaystyle \int_{S^n}\tilde H^{ij}\tilde{z}_i\tilde{z}_j\det(\tilde H)dS$, all other terms in the above are approximated by the term replacing $\tilde h$ by $h$, $\tilde{z}$ by $z=hu$. For example, we have, for $\epsilon_1$ sufficiently small, 
\[
\left|\int_{S^n}\tilde H^{ij}\delta_{ij}\tilde{z}^2 \det{\tilde H}dS-\int_{S^n} H^{ij}\delta_{ij}z^2 \det{ H}dS\right|\leq C\epsilon.
\]
To see the convergence of the first term, we use the identity $\sqrt{-1}\partial z\wedge\bar{\partial}z=\frac{1}{2}(\ddc z^2-z\ddc z)$ to rewrite it as the sum of two parts:\[
\begin{aligned}
    \int_{S^n}\tilde{H}^{ij}\tilde{z}_i\tilde{z}_j\det(\tilde{H})dS=&\int_{M}f(r)\sqrt{-1}\partial r\wedge\bar{\partial}r\wedge \sqrt{-1}\partial \tilde{z}\wedge\bar{\partial}\tilde{z}(\sqrt{-1}\ddc \tilde{h})^{n-1}\\
    =&\int_{M}f(r)\sqrt{-1}\partial r\wedge\bar{\partial}r\wedge\frac{1}{2}(\ddc \tilde{z}^2-\tilde{z}\ddc \tilde{z})\wedge(\sqrt{-1}\ddc \tilde{h})^{n-1}\\
    =:&III+IV
\end{aligned}
\]
Since $\tilde{h}\rightarrow h$ in $L^\infty$, we know that $\tilde{z}\rightarrow hu$ in $L^\infty$. Here $u$ is $C^2$ function, and $h$ is a plurisubharmonic function. Since $h$ is uniformly bounded, and it's the support function of a convex body, we have $|D h|$ is also uniformly bounded. It follows that \[
\ddc(hu)=u\ddc h+\sqrt{-1}\partial u\wedge\bar{\partial}h+\sqrt{-1}\partial{h}\wedge\bar{\partial}u+h\ddc u\geq -C_0\beta,
\] for some constant $C_0$ depending on $||u||_{C^2},||h||_{L^\infty}$. It means that $hu$ is a quasiplurisubharmonic function. Now the proof of the convergence of the integral $III,IV$ is similar to the estimate of $I$ as before, and hence, we get \[
\int_{S^n}\tilde{H}^{ij}z_iz_j\det(\tilde{H})dS\rightarrow \int_{M}f(r)\sqrt{-1}\partial r\wedge\bar{\partial}r\wedge\frac{1}{2}(\ddc z^2-z\ddc z)\wedge(\sqrt{-1}\ddc h)^{n-1}.
\] Therefore, for any $C^2$-function $\hat{z}$ with $|\hat{z}-z|_{L^\infty}\leq \epsilon_2(\epsilon,u,h)$, by the Chern-Levine-Nirenberg inequality, we have \[\left|
\int_{S}H^{ij}\hat{z}_i\hat{z}_j\det HdS-\int_{M}f(r)\sqrt{-1}\partial r\wedge\bar{\partial}r\wedge\frac{1}{2}(\ddc z^2-z\ddc z)\wedge(\sqrt{-1}\ddc h)^{n-1}\right|\leq \epsilon
\] Again, all other terms in $\tilde{\lambda}$ can be approximated by replace $\tilde{z}$ by $\hat{z}$. 
It follows that \[
\tilde{\lambda}\geq \frac{\lambda_{1,e}(K)-C\epsilon}{1+C\epsilon}.
\]Hence, we get
$\lambda_{1,e}(K)\leq \lambda_{1,e}(-L_K)$. The other direction is similar, and this proves that $\lambda_{1,e}(-L_K)=\lambda_{1,e}(K).$
\end{proof}
Our definition $\lambda_{1, e}(-L_K)$, in particular the formulation \eqref{eigenvalue-003}, is essentially the same  with $\lambda^C$ introduced in \cite{Put, Milman22}. There are two other natural notions \cite{KM01, Milman22} to extend the definition to nonsmooth convex bodies. Let $K_i$ be  any sequence of convex bodies in $\cK^2_{+, e}$ converging to $K$, then one can introduce
\[
\underline{\lambda}_{1, e}(-L_K)=\liminf_{K_i\rightarrow K} \lambda_{1, e}(-L_{K_i}),\; \overline{\lambda}_{1, e}(-L_K)=\limsup_{K_i\rightarrow K} \lambda_{1, e}(-L_{K_i}).
\]
It seems to be a very nontrivial problem whether, for a nonsmooth convex body $K\in \cK_e$ and $n\geq 2$,
\[
\underline{\lambda}_{1, e}(-L_K)=\overline{\lambda}_{1, e}(-L_K)=\lambda_{1, e}(-L_K).
\]

Next, we use those notations to compute the second variational formula of $F$ when the convex body is nonsmooth. Recall that $$F_{d\mathfrak{m}}(h):=\frac{1}{p}\log \int_{S^n}h^pd\mathfrak{m}-\frac{1}{n+1}\log V(h).$$ Here $d\mathfrak{m}$ is a Radon measure, and $h=h_t=h(t)$ is a family of convex bodies, and suppose that $h(t):[-1,1]\rightarrow C^0(S^n)$ is a $C^2$ map. Indeed, our previous argument for the existence of a measure $U^{ij}u_iu_jdS$ for a nonsmooth convex body shows that we can express the integral of a function $g$ on $S^n$ as an integral on $\mathbb{R}^{n+1}\times\mathbb{T}^{n+1}$,\begin{lemma}
    Let $g$ be a homogeneous function of degree $0$ : $\mathbb{R}^{n+1}\rightarrow \mathbb{R}$, \[\int_{S^n}gdS_h=\int_{\mathbb{R}^{n+1}\times\mathbb{T}^{n+1}}\chi(r)g\sqrt{-1}\partial r\wedge\bar{\partial}r\wedge(\ddc h)^{n},
\]where in the second integral we extend $g$ trivially on each $\mathbb{T}^{n+1}$ fiber. 
\end{lemma}

Let's prove Aleksandrov's lemma first under our complex language. \[
\begin{aligned}
   \frac{V(h_t)-V(h_{t_0})}{t-t_0}=&\frac{1}{n+1}\int_{M}r^{-1}\chi(r)\frac{h_t-h_{t_0}}{t-t_0}(\ddc h_{t_0})^{n}\wedge\sqrt{-1}\partial{r}\wedge \bar{\partial} r\\
    &+
    \frac{1}{n+1}\int_{M}r^{-1}\chi(r)h_t\frac{(\ddc h_t)^n-(\ddc h_{t_0})^n}{t-t_0}\wedge\sqrt{
    -1}\partial{r}\wedge \bar{\partial} r\\
\end{aligned}
\]
Let $\Theta_t:=\sum_{k=1}^{n-1}(\ddc h_t)^k(\ddc h_{t_0})^{n-1-k}$, then the second term can be written as \[
\int_{M}r^{-1}\chi(r)\frac{h_t-h_{t_0}}{t-t_0}\ddc h_t\wedge\Theta_t\wedge\sqrt{-1}\partial{r}\wedge\bar{\partial}r.
\]
So if $\frac{h_t-h}{t-t_0}\rightarrow h_{t_0}'$ in $L^\infty$ as $t\rightarrow 0$, we get $\ddc h_t\wedge\Theta_t\rightarrow n(\ddc h_{t_0})^n$ as currents, which implies 
\[
\lim_{t\rightarrow t_0}\frac{V(h_t)-V(h_{t_0})}{t-t_0}=\int_{M}r^{-1}\chi(r)h_{t_0}'(\ddc h_{t_0})^n\wedge \sqrt{-1}\partial{r}\wedge \bar{\partial} r=\int_{S^n}h_{t_0}'dS_{h_{t_0}}.
\]

So the first variation of $F$ becomes \[
\frac{d}{dt}F(h_t)=\int_{S^n}h_t'\frac{h^{p-1}d\mathfrak{m}}{\int_{S^n}h^pd\mathfrak{m}}-\int_{S^n}h_t'\frac{dS_{h_t}}{(n+1)V(h_t)}=:F_1-F_2.
\]

Now, we examine the second variation. For the first term $F_1$, it's straightforward that \[
\frac{d}{dt}F_1=\int_{S^n}h_t''\frac{h^{p-1}d\mathfrak{m}}{\int_{S^n}h^pd\mathfrak{m}}+\frac{p-1}{\int_{S^n}h^pd\mathfrak{m}}\int_{S^n}h_t'^2h^{p-2}d\mathfrak{m}-p\frac{(\int_{S^n}h_t'h^{p-1}d\mathfrak{m})^2}{(\int_{S^n}h^pd\mathfrak{m})^2}.
\]
For the second term $F_2$:\[
\begin{aligned}
    \frac{F_2(h_t)-F_2(h_0)}{t}=&\int_Mh_t'r^{-1}\chi(r)(\ddc h_t)^n\wedge\sqrt{-1}\partial{r}\wedge\bar{\partial{r}}\frac{1}{t}\left(\frac{1}{(n+1)V(h_t)}-\frac{1}{(n+1)V(h_0)}\right)\\
    &+\frac{1}{(n+1)V(h_0)}\int_Mr^{-1}\chi
    (r)\frac{h_t'-h_0'}{t}(\ddc h_t)^n\wedge\sqrt{-1}\partial r\wedge\bar{\partial}r\\
    &+\frac{1}{(n+1)V(h_0)}\int_Mr^{-1}\chi(r)h_0'\frac{(\ddc h_t)^n-(\ddc h_0)^n}{t}\sqrt{-1}\partial r\wedge\bar{\partial}r\\
    &=(a)+(b)+(c)
\end{aligned}
\]It's easy to see that as $t\rightarrow 0$, $(a)\rightarrow \frac{1}{V^2(h_0)}\left(\int_{S^n}h_0'dS_{h_0}\right)^2$, and $(b)\rightarrow \frac{1}{(n+1)V(h_0)}\int_{S^n}h_0''dS_{h_0}$. 

We assume that $h_0'$ is a quasi-plurisubharmonic function(here, a quasi-plurisubharmonic function means the sum of it with a smooth function is plurisubharmonic), then integration by parts gives \[
\begin{aligned}
    (c)=&\frac{1}{(n+1)V(h_0)}\int_Mr^{-1}\chi(r)\frac{h_t-h_0}{t}\ddc h_0'\wedge \Theta_t\wedge\sqrt{-1}\partial r\wedge\bar{\partial }r.
\end{aligned}
\] Since $\ddc h_0'\wedge \Theta_t\rightarrow \ddc h_0'\wedge(\ddc h_0)^{n-1}$, we get \[(c)\rightarrow\frac{1}{(n+1)V(h_0)} \int_Mnr^{-1}\chi(r)h_0'\ddc h_0'\wedge(\ddc h_0)^{n-1}\wedge\sqrt{-1}\partial r\wedge\bar{\partial}r \quad \text{as current when }t\rightarrow 0.\] 

Recall that a nonsmooth convex body with support function $h$ is a critical point of $F$ if \[\int_{S^n}g
\frac{h^{p-1}d\mathfrak{m}}{\int_{S^n}h^pd\mathfrak{m}}=\int_{S^n}g\frac{dS_{h}}{(n+1)V(h)},\forall g\in C^0(S^n).
\]
So if $h_0$ is a critical point of $F$, we have \[
\int_{S^n}h_0''\frac{h_0^{p-1}d\mathfrak{m}}{\int_{S^n}h_0^pd\mathfrak{m}}=\int
_{S^n}h_0''\frac{dS_{h_0}}{(n+1)V(h_0)},
\]
and \[\frac{p-1}{\int_{S^n}h_0^pd\mathfrak{m}}
\int_Mh_0'^2h_0^{p-2}d\mathfrak{m}=\frac{p-1}{(n+1)V(h_0)}\int_{S^n}h_0'^2h_0^{-1}dS_{h_0}.
\]
It follows \[
\begin{aligned}
 \frac{d^2}{dt^2}|_{t=0}F(h_t)=& \frac{p-1}{(n+1)V(h_0)}\int_{S^n}h_0'^2h_0^{-1}dS_{h_0} +\frac{(n+1-p)}{(n+1)V(h_0)^2}\left(\int_{S^n}h_0'dS_{h_0}\right)^2\\
 &-\frac{1}{(n+1)V(h_0)}\int_Mnh_0'r^{-1}\chi(r)(\ddc h_0')\wedge(\ddc h_0)^{n-1}\wedge\sqrt{-1}\partial r\wedge\bar{\partial} r
\end{aligned}
\]

Assume that $h_0'=h_0z$, and $z$ is a bounded quasi-plurisubharmonic function with bounded gradient. Then $h_0'$ is a quasi-plurisubharmonic function.  We claim that for any smooth function $f(r)$, \[\begin{aligned}
    &\int_{M}f(r)h_0z\ddc (h_0z)\wedge(\ddc h_0)^{n-1}\sqrt{-1}\partial r\wedge\bar{\partial r}\\
    =&\int_Mf(r)h_0^2\frac{1}{2}(\ddc z^2-z\ddc z)\wedge(\ddc h_0)^{n-1}\sqrt{-1}\partial r\wedge\bar{\partial r}\\
    &+\int_M f(r)h_0z^2(\ddc h_0)^{n-1}\sqrt{-1}\partial r\wedge\bar{\partial r}.
\end{aligned}
\]The equality is straightforward when both $h_0$ and $z$ is smooth, and the quasi-plurisubharmonic case is just a consequence of approximation using Chern-Levine-Nirenberg inequality as previously. Later, we just write $2\sqrt{-1}\partial z\wedge\bar{\partial z}=\ddc z^2+z\ddc z$ when $z$ is only quasi-plurisubharmonic, and similarly, we define the Radom measure $U^{ij}z_iz_jdS$ with this understanding. Hence  \[\begin{aligned}
    &\int_M nr^{-1}\chi(r)h_0'(\ddc h_0')(\ddc h_0)^{n-1}\sqrt{-1}\partial r\wedge \bar{\partial}r\\
=&\int_{S^n}H_0^{ij}z_iz_jh_0^2dS_{h_0}+n\int_{S^n}z^2h_0dS_{h_0}
\end{aligned} \]

So with those notations,  one gets:\begin{prop}
    Let $h_t:[0,1]\to C(S^n)$ be a $C^2$-family of support functions of convex bodies such that $h_0'=h_0z$ for some quasi-plurisubharmonic function $z$ with bounded gradient, and assume that $h_0$ is a critical point of $F_{d\mathfrak{m}}(h)$ for a Radon measure $d\mathfrak{m}$ on $S^n$, then \begin{equation}
\begin{aligned}\label{nonsmoothsvf}
    \frac{d^2}{dt^2}|_{t=0}F&=\frac{p-1-n}{(n+1)V(h_0)}\int_{S^n}z^2h_0dS_{h_0}+\frac{(n+1-p)}{(n+1)V(h_0)^2}\left(\int_{S^n}h_0'dS_{h_0}\right)^2\\
    &-\frac{1}{(n+1)V(h)}\int_{S^n}H_0^{ij}z_iz_jh_0^2dS_{h_0}.
\end{aligned}
\end{equation}
\end{prop}
Here, a $C^2$-family of support functions means \[
\lim_{s\rightarrow0}\frac{h_{t+s}-h_t}{s}\to h'_t  \quad \text{and} \quad \lim_{s\rightarrow0}\frac{h_{t+s}'-h_t'}{s}\to  h_t'' \quad \text{in } L^\infty\]

\end{document}